\documentclass[11pt]{amsart}
\usepackage{amssymb}
\usepackage{amsmath}
\usepackage{amsfonts}
\usepackage{graphicx}

\usepackage{hyperref}
    \usepackage{aeguill}
    \usepackage{type1cm}

\theoremstyle{plain}
\newtheorem{thm}{Theorem}[section]

\newtheorem{claim}[thm]{Claim}

\newtheorem{corollary}[thm]{Corollary}

\newtheorem{definition}[thm]{Definition}
\newtheorem{example}[thm]{Example}

\newtheorem{lemma}[thm]{Lemma}

\newtheorem{proposition}[thm]{Proposition}
\newtheorem{remark}[thm]{Remark}

\numberwithin{equation}{section}
\newcommand{\N}{\mathbb{N}}

\newcommand{\B}{\mathbb{B}}
\newcommand{\E}{\mathcal{E}}
\newcommand{\R}{\mathbb{R}}
\newcommand{\C}{\mathbb{C}}

\DeclareMathOperator{\co}{co\,\!}

\DeclareMathOperator{\lip}{Lip\,\!}

\usepackage[usenames,dvipsnames]{color}

\usepackage[dvipsnames]{xcolor}

\begin{document}
\title[$C^{1,1}$ convex extensions of convex functions]{An Extension Theorem for convex functions of class $C^{1,1}$ on Hilbert spaces}
\author{Daniel Azagra}
\address{ICMAT (CSIC-UAM-UC3-UCM), Departamento de An{\'a}lisis Matem{\'a}tico,
Facultad Ciencias Matem{\'a}ticas, Universidad Complutense, 28040, Madrid, Spain }
\email{azagra@mat.ucm.es}

\author{Carlos Mudarra}
\address{ICMAT (CSIC-UAM-UC3-UCM), Calle Nicol\'as Cabrera 13-15.
28049 Madrid SPAIN}
\email{carlos.mudarra@icmat.es}

\date{February 17, 2016}


\thanks{D. Azagra was partially supported by Ministerio de Educaci\'on, Cultura y Deporte, Programa Estatal de Promoci\'on del Talento y su Empleabilidad en I+D+i, Subprograma Estatal de Movilidad. C. Mudarra was supported by Programa Internacional de Doctorado Fundaci\'on La Caixa--Severo Ochoa. Both authors partially supported by MTM2012-34341.}


\begin{abstract}
Let $\mathbb{H}$ be a Hilbert space, $E \subset \mathbb{H}$ be an arbitrary subset and $f: E \to \R, \: G: E \to \mathbb{H}$ be two functions. We give a necessary and sufficient condition on the pair $(f,G)$ for the existence of a \textit{convex} function $F\in C^{1,1}(\mathbb{H})$ such that $F=f$ and $\nabla F =G$ on $E$. We also show that, if this condition is met, $F$ can be taken so that $\lip(\nabla F) = \lip(G)$. We give a geometrical application of this result, concerning interpolation of sets by boundaries of $C^{1,1}$ convex bodies in $\mathbb{H}$. Finally, we give a counterexample to a related question concerning smooth convex extensions of smooth convex functions with derivatives which are not uniformly continuous.
\end{abstract}

\maketitle

\section{Introduction and main results}

Throughout this paper $\mathbb{H}$ will be a real Hilbert space equipped with inner product $\langle \cdot \: , \cdot \rangle.$ The norm in $\mathbb{H}$ will be denoted by $\| \cdot \|.$ By a $1$-jet $(f,G)$ on a subset $E \subset \mathbb{H}$ we understand a pair of functions $f:E \to \R,\: G:E \to \mathbb{H}.$ 
Given a $1$-jet $(f,G)$ defined on $E \subset \mathbb{H},$ Le Gruyer proved in \cite{LeGruyer} that a necessary and sufficient condition on the jet $(f,G)$ for having a $C^{1,1}$ extension $F$ to the whole space $ \mathbb{H}$ is that
$$
\Gamma(f,G,E):=\sup_{x,y\in E} \left( \sqrt{A_{x,y}^2+B_{x,y}^2} + |A_{x,y}| \right) < \infty,
$$
where
$$
A_{x,y} = \frac{2(f(x)-f(y))+\langle G(x)+G(y),y-x \rangle}{\| x-y\|^2} \quad \text{and}
$$
$$
 B_{x,y} = \frac{\|G(x)-G(y)\|}{\|x-y\|} \quad \text{for all} \quad x,y\in E, x\neq y.
 $$
This condition is equivalent to
$$
2\sup_{y\in\mathbb{H}}\sup_{a\neq b\in E}\frac{f(a)-f(b)+\langle G(a), y-a\rangle - \langle G(b), y-b\rangle}{\|a-y\|^2+\|b-y\|^2} <\infty,
$$
and in particular Le Gruyer's Theorem provides a generalization (to the setting of Hilbert spaces) of Glaeser's version of the Whitney Extension Theorem \cite{Whitney, Glaeser} for $C^{1,1}$ functions. See also \cite{Wells} for another generalization of the $C^{1,1}$ version of the Whitney Extension Theorem for functions defined on subsets of the Hilbert space, and for a proof that the Whitney Extension Theorem fails for $C^3$ functions on Hilbert spaces. We also refer to \cite{JSSG} for a version of the Whitney Extension Theorem for $C^1$ functions on some Banach spaces (including Hilbert spaces).

Le Gruyer also shows in  \cite{LeGruyer} that the extension $F$ satisfies 
$$
\lip(\nabla F)=\Gamma(F, \nabla F,\mathbb{H}) = \Gamma(f,G,E).
$$  

Our purpose in this paper is to solve an analogous problem for convex functions. In a recent paper \cite{AM}, by introducing a new condition $(CW^{1,1})$, see Definition \ref{definitioncw11} below, we gave a satisfactory solution to this problem for convex functions of the class $C^{1,1}(\R^n)$ (in fact, for all the classes $C^{1,\omega}(\R^n)$, where $\omega$ is a modulus of continuity), with a good control of the Lipschitz constant of the gradient of the extension in terms of that of $G,$ namely $\lip(\nabla F) \leq c(n) \lip(G),$ where $c(n)$ only depends on $n$ (but tends to infinity with $n$); see also \cite{AMSmooth} for information about related problems of higher order. In this paper we generalize and improve this result for $C^{1,1}$ convex functions defined on an arbitrary Hilbert space, showing in particular that those constants $c(n)$ can all be taken equal to $1$. Nevertheless, it must be observed that whereas the proofs in \cite{AM} (and of course the proof of the $C^{1,1}$ version of the Whitney Extension Theorem too) are constructive, the proofs of Le Gruyer's Theorem in \cite{LeGruyer} and of the main result in the present paper (which is strongly inspired by that of Le Gruyer's) are not, as they both rely on an application of Zorn's lemma.

\begin{definition}\label{definitioncw11}
We will say that a pair of functions $f:E \to \R,\: G:E \to \mathbb{H}$ defined on a subset $E\subset \mathbb{H},$ satisfies condition $(CW^{1,1})$ on $E$ provided that there exists a constant $M>0$ with
$$
f(x)-f(y)-\langle G(y), x-y \rangle \geq \frac{1}{2M}\|G(x)-G(y)\|^2 
\eqno(CW^{1,1})
$$
for all $x,y\in E.$
\end{definition}

\begin{remark}\label{remarkcw11implieslipschitz}
If $(f,G)$ satisfies $(CW^{1,1})$ on $E,$ then 
$$
f(x) \geq f(y)+ \langle G(y), x-y \rangle \quad \text{for all} \quad x,y\in E
$$
and
$$
\sup_{x\neq y,\: x,y\in E} \left\lbrace \frac{|f(x)-f(y)-\langle G(y), x-y \rangle|}{\|x-y\|^2}, \frac{\|G(x)-G(y)\|}{\|x-y\|} \right\rbrace \leq M.
$$
In particular $G$ is $M$-Lipschitz on $E.$
\end{remark}
\begin{proof}
The first inequality is obvious. For the second one, given $x,y\in E,$ the condition $(CW^{1,1})$ gives us the inequalities:
$$
f(x)-f(y)-\langle G(y), x-y \rangle \geq \frac{1}{2M}\|G(x)-G(y)\|^2 \quad \text{and}
$$
$$
f(y)-f(x)-\langle G(x), y-x \rangle \geq \frac{1}{2M}\|G(y)-G(x)\|^2,
$$
the sum of which yields
$$
 \langle G(x)-G(y), x-y \rangle \geq \frac{1}{M} \|G(x)-G(y)\|^2,
 $$
which in turn implies $\|G(x)-G(y)\| \leq M\|x-y\|.$ 
On the other hand, using again $(CW^{1,1})$ we obtain
$$
0 \leq f(x)-f(y)-\langle G(y), x-y \rangle \leq \langle G(y)-G(x),y-x \rangle.
$$
The desired inequality follows by combining the last one with the fact that $G$ is $M$-Lipschitz on $E.$
\end{proof}

The main result of this paper is as follows. 
\begin{thm} \label{mainextensiontheorem}
Let $E$ be a subset of $ \mathbb{H}$ and $f:E \to \R,\: G : E \to \mathbb{H}$ be two functions. Then there exists a convex function $F$ of class $C^{1,1}(\mathbb{H})$ such that $F=f$ and $\nabla F = G$ on $E$ if and only if $(f,G)$ satisfies condition $(CW^{1,1})$ on $E.$ Moreover, if $M>0$ is as in Definition \ref{definitioncw11}, then $F$ can be taken such that $(F,\nabla F)$ also satisfies $(CW^{1,1})$ on $\mathbb{H}$ with the same constant $M.$
\end{thm}

Equivalently, bearing in mind Remark \ref{remarkcw11implieslipschitz}, Theorem \ref{mainextensiontheorem} can be reformulated in terms of the Lipschitz constant as follows.

\begin{thm}\label{lipschitzversionmainextensiontheorem}
Let $E$ be a subset of $ \mathbb{H},\: f:E \to \R$ be a function and $G : E \to \mathbb{H}$ a nonconstant Lipschitz mapping. A necessary and sufficient condition on the pair $(f,G)$ for the existence of a convex function $F$ of class $C^{1,1}(\mathbb{H})$ such that $F=f$ and $\nabla F = G$ on $E$ is that $(f,G)$ satisfies condition $(CW^{1,1})$ on $E$ with $M=\lip(G).$ In addition, if this condition is met, $F$ can be taken such that $\lip(\nabla F) =\lip (G).$ 
\end{thm}
Obviously, there is no loss of generality in assuming that $G$ is not constant, as the problem is trivial otherwise (a $1$-jet $(f,G)$ on $E$ satisfying $f(x)-f(y) -\langle G(y), x-y\rangle\geq 0$ for $x,y\in E$ and such that $G$ constant extends to an affine function on $\mathbb{H}$).

As in \cite{AM}, we can use the above results to solve a geometrical problem concerning characterizations of subsets of a Hilbert space which can be interpolated by
boundaries of $C^{1,1}$ convex bodies (with prescribed unit outer normals). Namely, if
$C$ is a subset of a Hilbert space $\mathbb{H}$ and we are given a Lipschitz map $N:C\to \mathbb{H}$ such that $|N(y)|=1$ for every $y\in C$, it is natural to ask what conditions on $C$ and $N$ are necessary and sufficient for $C$ to be a subset of the boundary of a $C^{1,1}$ convex body $V$ such that $0\in\textrm{int}(V)$ and $N(y)$ is outwardly normal to $\partial V$ at $y$ for every $y\in C$. A suitable set of conditions is:
\begin{align*}
& (\mathcal{O}) & \langle N(y), y \rangle \geq \delta \textrm{ for all } y\in C; \\
& (\mathcal{K}\mathcal{W}^{1,1}) & \langle N(y), y-x\rangle \geq \delta |N(y)-N(x)|^2 \textrm{ for all } x, y\in C,
\end{align*}
for some $\delta>0$. The proof of \cite[Theorem 1.5]{AM} can easily be adapted to obtain the following.
\begin{corollary}\label{corollary for C11 convex bodies}
Let $C$ be a subset of a Hilbert space $\mathbb{H}$, and let $N:C\to\mathbb{H}$ be a Lipschitz mapping such that $|N(y)|=1$ for every $y\in C$.
Then the following statements are equivalent:
\begin{enumerate}
\item There exists a $C^{1,1}$ convex body $V$ with $0\in\textrm{int}(V)$ and such that $C\subseteq \partial V$ and $N(y)$ is outwardly normal to $\partial V$ at $y$ for every $y\in C$.
\item  $C$ and $N$ satisfy conditions $(\mathcal{O})$ and  $(\mathcal{K}\mathcal{W}^{1,1})$ for some $\delta>0$.
\end{enumerate}
\end{corollary}

It is natural to look for analogues of Theorems \ref{mainextensiontheorem} and \ref{lipschitzversionmainextensiontheorem} for $1$-jets $(f,G)$ on a closed subset $C$ of a Hilbert space $\mathbb{H}$ with $G$ not necessarily Lipschitz. If $G$ is uniformly continuous, it seems plausible that the condition $(CW^{1, \omega})$ found in \cite{AM} may be necessary and sufficient for $(f,G)$ to have a $C^{1, \omega}$ extension to $\mathbb{H}$. However, the proofs in the present paper cannot be adapted to that purpose. On the other hand, for the method of proof of \cite{AM} to work in an infinite-dimensional setting, we would need to have, among other things, a $C^{1, \omega}$ version of Whitney's extension theorem valid for infinite-dimensional Hilbert spaces, and to the best of our knowledge no one has established such a result (with the exception of Wells and Le Gruyer \cite{Wells, LeGruyer} in the particular case that $\omega(t)=t$). What we do know is that the conditions $(C)$, $(CW^1)$ and $(W^1)$ of \cite[Theorem 1.7]{AM} are not sufficient in the infinite-dimensional setting because, as we will show in Example \ref{counterexample in Hilbert} below, there exist bounded, smooth convex functions defined on an open neighborhood of a closed ball in $\mathbb{H}$ which have no continuous convex extensions to all of $\mathbb{H}$.

\section{Proof of Theorem \ref{mainextensiontheorem}}

\subsection{Necessity}
The necessity of condition $(CW^{1, 1})$ in Theorem \ref{mainextensiontheorem} follows from the following Proposition.
\begin{proposition}\label{necessity}
Let $f\in C^{1, 1}(\mathbb{H})$ be convex, and assume that $f$ is not affine. Then
$$
f(x)-f(y)-\langle \nabla f(y), x-y\rangle \geq \frac{1}{2M} \|\nabla f(x)-\nabla f(y)\|^2
$$
for all $x, y\in \mathbb{H}$, where 
$$
M=\sup_{x, y\in \mathbb{H}, \, x\neq y}\frac{\|\nabla f(x)-\nabla f(y)\|}{\|x-y\|}.
$$
\end{proposition}
\noindent On the other hand, if $f$ is affine, it is obvious that $(f, \nabla f)$ satisfies $(CW^{1,1})$ on every $E\subset \mathbb{H}$, for every $M>0$. 
\begin{proof}
Suppose that there exist different points $x, y\in \mathbb{H}$ such that
$$
f(x)-f(y)-\langle \nabla f(y), x-y\rangle < \frac{1}{2M} \|\nabla f(x)-\nabla f(y)\|^2,
$$
and we will get a contradiction.

\noindent {\bf Case 1.} Assume further that $M=1$, $f(y)=0$, and $\nabla f(y)=0$.
By convexity this implies $f(x)\geq 0$.
Then we have 
$$
0\leq f(x)<\frac{1}{2}\|\nabla f(x)\|^2.
$$
Call $a=\|\nabla f(x)\|>0$, $b=f(x)$, set 
$$
v=-\frac{1}{\|\nabla f(x)\|}\nabla f(x),
$$
and define 
$$
\varphi(t)=f(x+tv)
$$
for every $t\in\R$. We have $\varphi(0)=b$, $\varphi'(0)=-a$, and  $\varphi'$ is $1$-Lipschitz. This implies that 
$$
|\varphi(t)-b+at|\leq \frac{t^2}{2}
$$
for every $t\in\R^{+}$, hence also that 
$$
\varphi(t)\leq -at+b+\frac{t^2}{2} \textrm{ for all } t\in \R^{+},
$$
By assumption we have $
b<\frac{1}{2} a^2,
$
and therefore
$$
f\left( x+a v\right)=\varphi\left( a\right)\leq -a^2+b+\frac{a^2}{2}<0,
$$
which is in contradiction with the assumptions that $f$ is convex, $f(y)=0$, and $\nabla f(y)=0$. This shows that
$$
f(x)\geq \frac{1}{2}\|\nabla f(x)\|^2.
$$

\noindent {\bf Case 2.} Assume only that $M=1$. Define
$$
g(z)=f(z)-f(y)-\langle \nabla f(y), z-y\rangle
$$
for every $z\in \mathbb{H}$. Then $g(y)=0$ and $\nabla g(y)=0$. By Case 1, we get
$$
g(x)\geq \frac{1}{2}\|\nabla g(x)\|^2,
$$
and since $\nabla g(x)=\nabla f(x)-\nabla f(y)$ the Proposition is thus proved in the case when $M=1$.

\noindent {\bf Case 3.} In the general case, we may assume $M>0$ (the result is trivial for $M=0$). Consider $\psi=\frac{1}{M}f$, which satisfies the assumption of Case 2. Therefore
$$
\psi(x)-\psi(y)-\langle \nabla \psi(y), x-y\rangle \geq \frac{1}{2} \|\nabla \psi(x)-\nabla \psi(y)\|^2,
$$
which is equivalent to the desired inequality.
\end{proof}
\subsection{Sufficiency}
Now we assume that $(f,G)$ satisfies condition $(CW^{1,1})$ on the set $E\subset \mathbb{H}$ with constant $M>0.$ If we prove that for any $x\in \mathbb{H} \setminus E$ there exist $z_x\in \R$ and $Z_x\in \mathbb{H}$ such that the pair $(\widetilde{f}, \widetilde{G}),$ defined by $\widetilde{f}=f,\: \widetilde{G}=G$ on $E$ and $\widetilde{f}(x)=z_x,\: \widetilde{G}(x)=Z_x,$ satisfies $(CW^{1,1})$ on $E \cup \{ x \}$ with the same constant $M,$ then 
Zorn's Lemma will imply the existence of a pair $(F, \tilde{\nabla} F)$ satisfying $(CW^{1,1})$ on $\mathbb{H}$ with constant $M.$ Hence by Remark \ref{remarkcw11implieslipschitz}, $F$ will be a convex function of class $C^{1,1}(\mathbb{H})$ such that $F=f,\: \nabla F = G$ on $E$ and $\lip(\nabla F) \leq M$, and this simultaneously will complete the proofs of Theorems \ref{mainextensiontheorem} and \ref{lipschitzversionmainextensiontheorem}.

To sum up, our only assumption is
\begin{equation}\label{hypothesisZornslemma}
f(b)-f(a)-\langle G(a), b-a \rangle \geq \frac{1}{2M}\|G(a)-G(b)\|^2 \quad \text{for all} \quad a,b\in E,
\end{equation}
and we have to show that for every $x\in \mathbb{H} \setminus E$ there exist $f(x) \in \R$ and $G(x) \in \mathbb{H}$ (denoted above by $z_x$ and $Z_x$ respectively) such that
\begin{align*}
&f(x)-f(a)-\langle G(a), x-a \rangle \geq \frac{1}{2M} \|G(x)-G(a)\|^2  \quad \text{and} \quad \\
& f(a)-f(x)-\langle G(x),a-x \rangle \geq \frac{1}{2M} \|G(x)-G(a)\|^2 \quad \text{for all} \quad a\in E.
\end{align*}
Note that these conditions are equivalent to 
\begin{align*}
& f(x) \geq f(a)+ \langle G(a), x-a \rangle + \frac{1}{2M} \|G(x)-G(a)\|^2 \,\,\, \textrm{ and }\\
& f(x) \leq f(b)- \langle G(x), b-x \rangle - \frac{1}{2M}  \| G(x)-G(b)\|^2 \quad \text{for all} \quad a,b\in E.
\end{align*}
If we prove the existence of a vector $G(x)\in \mathbb{H}$ such that
\begin{align*}
& s(x) :=\sup_{a\in E} \left( f(a)+ \langle G(a), x-a \rangle + \frac{1}{2M}  \|G(x)-G(a)\|^2 \right) \\
&\leq I(x):= \inf_{b\in E} \left( f(b)- \langle G(x), b-x \rangle - \frac{1}{2M}  \| G(x)-G(b)\|^2 \right),
\end{align*}
then it will be enough for us to take $f(x)$ as any number in the interval $[ s(x) , I(x) ].$ 

In what follows we will essentially keep Le Gruyer's notation because, 
although our numbers $\alpha_{a,b}$, $\beta_{a,b}$, $\Phi((a,b), (c,d))$, etc,  are different from Le Gruyer's, they will play a similar role in the proof.
Inspired by a strategy in Le Gruyer's proof of \cite[Theorem 2.6]{LeGruyer}, we will express the condition $s(x) \leq I(x)$ in the following way.

\begin{lemma}\label{lemmadefinitionballs}
The inequality $s(x) \leq I(x)$ is equivalent to
$$
\| G(x)- Z_{a,b}\|^2 \leq \alpha_{a,b}+ \beta_{a,b}, \quad \text{for all} \quad a,b\in E,
$$
where
\begin{align*}
& \alpha_{a,b} := M \Big ( f(b)-f(a)- \langle G(a), b-a \rangle \Big ) -\tfrac{1}{2}\|G(a)-G(b)\|^2,\\
& \beta_{a,b} := \Big \| \tfrac{1}{2} \Big ( G(b)-G(a)+M (x-b) \Big )\Big \|^2, \\
& Z_{a,b}:= \tfrac{1}{2} \Big ( G(a)+G(b)+ M (x-b) \Big ) .
\end{align*}
\end{lemma}
\begin{proof}
 We have that $s(x) \leq I(x)$ if and only if, for all $a,b\in E$,
 $$
 f(a)+ \langle G(a), x-a \rangle + \frac{1}{2M}\|G(x)-G(a)\|^2 \leq f(b)- \langle G(x), b-x \rangle - \frac{1}{2M} \| G(x)-G(b)\|^2.
 $$
 Multiplying by $M$ we have that
 $$
 \frac{1}{2} \left( \|G(x)-G(a)\|^2 + \|G(x)-G(b)\|^2 \right) + M \langle G(x), b-x \rangle \leq M( f(b)-f(a) ) +M \langle G(a),a-x \rangle.
 $$
 Applying the Paralelogram Law to the left-side term we obtain
 \begin{align*}
 \frac{1}{4} & \left( \| 2 G(x)-G(a)-G(b)\|^2 + \|G(b)-G(a)\|^2 \right) + M \langle G(x), b-x \rangle \\
 & \qquad \leq M( f(b)-f(a) ) +M \langle G(a),a-x \rangle,
 \end{align*}
 or equivalently 
 \begin{align*}
 & \bigg \|  G(x)-\frac{G(a)+G(b)}{2} \bigg \|^2  + M \langle G(x), b-x \rangle \\
 &  \leq M( f(b)-f(a) ) +M \langle G(a),a-x \rangle - \frac{1}{4} \|G(b)-G(a)\|^2.
 \end{align*}
This can be written as
  \begin{align*}
  & \bigg \|  G(x)-\frac{G(a)+G(b)}{2} \bigg \|^2  -  2 \Big \langle G(x)-\frac{G(a)+G(b)}{2}, \frac{M}{2} (x-b) \Big \rangle + \frac{M^2}{4} \|x-b\|^2 \\
 &  \leq M( f(b)-f(a) ) +M \langle G(a),a-x \rangle - \frac{1}{4} \|G(b)-G(a)\|^2 \\
  & \qquad + 2 \Big \langle \frac{G(a)+G(b)}{2}, \frac{M}{2} (x-b) \Big \rangle + \frac{M^2}{4} \|x-b\|^2,
 \end{align*}
 which is equivalent to 
 \begin{align*}
  & \bigg \| \left( G(x)-\frac{G(a)+G(b)}{2} \right) - \frac{M}{2} (x-b) \bigg \|^2 \leq  \\
  & \leq M \Big ( f(b)-f(a)-\langle G(a), b-a \rangle \Big ) - \frac{1}{2}\|G(a)-G(b)\|^2+ M \langle G(a), b-a \rangle  \\
  &  + \Big \langle G(a)+G(b), \frac{M}{2} (x-b)  \Big \rangle + \frac{M^2}{4} \|x-b\|^2 + M \langle G(a), a-x \rangle + \frac{1}{4}\|G(a)-G(b)\|^2.
 \end{align*}
By the definition of $Z_{a,b}$ and $\alpha_{a,b}$ we obtain
 \begin{align*}
   & \| G(x)-Z_{a,b} \|^2 \leq \alpha_{a,b}+ M \langle G(a), b-a \rangle  + \Big \langle G(a)+G(b), \frac{M}{2} (x-b) \Big \rangle \\
  &  +  \frac{M^2}{4}\|x-b\|^2 + M \langle G(a), a-x \rangle + \frac{1}{4}\|G(a)-G(b)\|^2 \\ 
  & = \alpha_{a,b} +  \Big \langle G(b), \frac{M}{2} (x-b) \Big \rangle -  \Big \langle G(a), \frac{M}{2}(x-b) \Big \rangle + \frac{1}{4}\|G(a)-G(b)\|^2+ \frac{M^2}{4}\|x-b\|^2 \\
  & = \alpha_{a,b} + \frac{1}{4}\|G(a)-G(b)\|^2 + 2 \Big \langle \frac{G(b)-G(a)}{2}, \frac{M}{2} (x-b) \Big \rangle + \frac{M^2}{4}\|x-b\|^2 \\
  & = \alpha_{a,b} + \bigg \| \frac{1}{2}(G(b)-G(a))+ \frac{M}{2} (x-b) \bigg \|^2 = \alpha_{a,b}+ \beta_{a,b}.
 \end{align*}
\end{proof}

Note that, by condition \eqref{hypothesisZornslemma}, the number $\alpha_{a,b}$ of Lemma \ref{lemmadefinitionballs} is nonnegative. This allows us to introduce the radii $r_{a,b} := \sqrt{\alpha_{a,b}+ \beta_{a,b}}$ and the closed balls $\B_{a,b} := B\left( Z_{a,b}, r_{a,b} \right)$ centered at $Z_{a,b}$ and radius $r_{a,b},$ for every $(a,b) \in E^2.$ Hence Lemma \ref{lemmadefinitionballs} shows that our problem can be reduced to showing that
$$
\bigcap_{(a,b)\in E} \B_{a,b} \neq \emptyset,
$$
because in this case it would be enough to take $G(x)$ as any point in $\bigcap_{(a,b)\in E} \B_{a,b}.$ In fact, thanks to the weak compactness of the closed balls in $\mathbb{H},$ this is equivalent to prove
$$
\bigcap_{(a,b)\in F} \B_{a,b} \neq \emptyset \quad \text{for every finite subset} \quad F \subset E.
$$
 Thus, from now on we may and do assume that $E$ is finite. We now introduce some new notations:
\begin{align*}
& \Phi( (a,b), (c,d) ) := r_{a,b}^2+r_{c,d}^2-\|Z_{a,b}\|^2-\|Z_{c,d}\|^2 \quad \text{for all} \quad (a,b), (c,d) \in E^2,\\
& \gamma_1(a) :=G(a) , \quad \gamma_2(a) :=G(a)+ M (x-a)   \quad \text{for all} \quad a \in E.
\end{align*}
As in Le Gruyer's proof of \cite[Theorem 2.6]{LeGruyer}, a crucial step consists in showing an inequality concerning $\Phi( (a,b), (c,d) )$ and the funcions $\gamma_1, \: \gamma_2.$
\begin{lemma} \label{lemmacrucialinequality}
For every $(a,b),(c,d) \in E^2$ we have
$$
\Phi( (a,b), (c,d) ) \geq -  \langle  \gamma_1(a), \gamma_2(d) \rangle-\langle \gamma_1(c), \gamma_2(b) \rangle .
$$
\end{lemma}
\begin{proof}
Using that
\begin{align*}
& \alpha_{a,b} = M \Big ( f(b)-f(a)- \langle G(a), b-a \rangle\Big ) -\tfrac{1}{2}\|G(a)-G(b)\|^2 \,\,\, \textrm{ and }\\
& \alpha_{c,d} = M \Big ( f(d)-f(c)- \langle G(c), d-c \rangle \Big ) -\tfrac{1}{2}\|G(c)-G(d)\|^2
\end{align*}
we obtain
\begin{align*}
& \alpha_{a,b}+ \alpha_{c,d} \\
& = M \Big ( f(d)-f(a)- \langle G(a), d-a \rangle \Big ) -\tfrac{1}{2}\|G(a)-G(d)\|^2 (=\alpha_{a,d}) \\
& \quad + M \Big ( f(b)-f(c)- \langle G(c), b-c \rangle \Big) -\tfrac{1}{2}\|G(c)-G(b)\|^2 (= \alpha_{c,b}) \\
& \quad + M \Big ( \langle G(a), d-b \rangle + \langle G(c), b-d \rangle \Big) (=:\delta_1) \\
& \quad + \frac{1}{2} \big( -\|G(a)-G(b)\|^2-\|G(c)-G(d)\|^2 + \|G(a)-G(d)\|^2+\|G(c)-G(b)\|^2 \big ) (=:\delta_2)\\
& = \alpha_{a,d}+\alpha_{c,b}+\delta_1+\delta_2.
\end{align*}
Of course, because $(a,d), (c,b)\in E^2,$ condition \eqref{hypothesisZornslemma} implies that $\alpha_{a,d},\alpha_{c,b} \geq 0.$ As for $\delta_1$, we have that 
$$
\delta_1= M \langle G(a)-G(c), d-b \rangle.
$$
Computing term by term in $\delta_2$ we obtain
\begin{align*}
 \delta_2 &= \frac{1}{2} \bigg ( -\|G(a)\|^2-\|G(b)\|^2+ 2 \langle G(a), G(b) \rangle -\|G(c)\|^2-\|G(d)\|^2+ 2 \langle G(c), G(d) \rangle  \\
& \qquad \quad +\|G(a)\|^2+\|G(d)\|^2- 2 \langle G(a), G(d) \rangle +\|G(c)\|^2+\|G(b)\|^2- 2 \langle G(c), G(b) \rangle \bigg )\\
& = \langle G(a), G(b) \rangle + \langle G(c), G(d) \rangle- \langle G(a), G(d) \rangle - \langle G(c), G(b) \rangle.
\end{align*} 
\begin{claim}\label{claimdeltas} We have that
$$
\delta_1+ \delta_2 = - \langle \gamma_1(a),\gamma_2(d) \rangle - \langle  \gamma_1(c), \gamma_2(b) \rangle  + \langle \gamma_1(a), \gamma_2(b)\rangle + \langle \gamma_1(c), \gamma_2(d)  \rangle.
$$
\end{claim}
Indeed, computing the term in the right side we obtain
\begin{align*}
&\quad - \Big \langle G(a),G(d) + M (x-d) \Big \rangle- \Big \langle  G(c),G(b) + M (x-b) \Big \rangle  \\
& \quad + \Big \langle G(a),G(b) + M (x-b) \Big \rangle + \Big \langle G(c),G(d) + M (x-d) \Big \rangle \\
& = -\langle G(a), G(d)- \langle G(b), G(c) \rangle + \langle G(a), G(b) \rangle + \langle G(c), G(d) \rangle \\
& \quad +M  \Big ( -\langle G(c), x-b \rangle - \langle G(a), x-d \rangle + \langle G(a), x-b \rangle + \langle G(c) , x-d \rangle \Big ) \\
& = \delta_2 + M \Big( \langle G(c), b-d \rangle + \langle G(a) , d-b \rangle \Big) = \delta_2+ \delta_1,
\end{align*}
and this proves our Claim.
\newline
On the other hand we note that
$$
\gamma_2(b)- \gamma_1(a) =  G(b)-G(a) + M (x-b),$$
and therefore
$$
\beta_{a,b}  = \Big \| \tfrac{1}{2} \Big ( G(b)-G(a)+M (x-b) \Big )\Big \|^2 = \Big \|\frac{\gamma_1(a)-\gamma_2(b)}{2} \Big \|^2.
$$
Similarly we have $\beta_{c,d} = \Big \| \frac{\gamma_1(c)-\gamma_2(d)}{2} \Big \|^2.$ We also see that 
\begin{equation}\label{sumgammasvab}
\gamma_1(a) +\gamma_2(b) = G(a)+G(b) + M (x-b) = 2 Z_{a,b}
\end{equation}

and $\gamma_1(c)+ \gamma_2(d)= 2 Z_{c,d}.$ These equations show that
\begin{align*}
& \beta_{a,b}+\beta_{c,d} = \Big \|\frac{\gamma_1(a)-\gamma_2(b)}{2} \Big \|^2 + \Big \|\frac{\gamma_1(c)-\gamma_2(d)}{2} \Big \|^2 \\
& \|Z_{a,b}\|^2+\|Z_{c,d}\|^2 = \Big \|\frac{\gamma_1(a)+\gamma_2(b)}{2} \Big \|^2 + \Big \|\frac{\gamma_1(c)+\gamma_2(d)}{2} \Big \|^2.
\end{align*}
By subtracting the second equation from the first one we obtain
$$
\beta_{a,b}+\beta_{c,d} -\|Z_{a,b}\|^2-\|Z_{c,d}\|^2 = - \langle \gamma_1(a), \gamma_2(b) \rangle - \langle \gamma_1(c), \gamma_2(d) \rangle .
$$
Finally, by using first Claim \ref{claimdeltas} and then the preceding equation we deduce
\begin{align*}
\Phi( (a,b), (c,d) )& = \alpha_{a,b}+\alpha_{c,d}+ \beta_{a,b}+\beta_{c,d} -\|Z_{a,b}\|^2-\|Z_{c,d}\|^2 \\
& = \alpha_{a,d}+\alpha_{c,b}+\delta_1+\delta_2+ \beta_{a,b}+\beta_{c,d} -\|Z_{a,b}\|^2-\|Z_{c,d}\|^2 \\
& \geq \delta_1+\delta_2+ \beta_{a,b}+\beta_{c,d} -\|Z_{a,b}\|^2-\|Z_{c,d}\|^2 \\
& = - \langle  \gamma_1(a), \gamma_2(d) \rangle-\langle \gamma_1(c), \gamma_2(b) \rangle .
\end{align*}
\end{proof}

In order to establish that $\bigcap_{(a,b) \in E^2} \B_{a,b} \neq \emptyset$ we first have to study the situation in which at least one of the balls of this family is a singleton. 

\begin{lemma}\label{lemmasingleton}
Suppose that there is $(a,b) \in E^2$ with $r_{a,b}=0$. Then 
$$
\bigcap_{(c,d) \in E^2} \B_{c,d} =\{ Z_{a,b}  \}
$$
and, in particular, the intersection is nonempty. 
\end{lemma}
\begin{proof}
The hyphotesis $r_{a,b}=0$ in particular implies that 
$$
0=\beta_{a,b} = \Big \| \tfrac{1}{2} \Big (G(b)-G(a)+M (x-b) \Big )\Big \|^2,
$$ and then we must have
$$
\gamma_1(a) =G(a) = G(b) + M (x-b) = \gamma_2(b).
$$
Because $2 Z_{a,b}=\gamma_1(a)+\gamma_2(b)$ (see equation \eqref{sumgammasvab}) we have that $Z_{a,b}= \gamma_1(a)$ and similarly $2 Z_{c,d}=\gamma_1(c)+\gamma_2(d).$ Combining this with the inequality of Lemma \ref{lemmacrucialinequality} we deduce, for all $(c,d) \in E^2,$
\begin{align*}
\Phi((a,b),(c,d)) & \geq  - \langle  \gamma_1(a), \gamma_2(d) \rangle-\langle \gamma_1(c), \gamma_2(b) \rangle  = - \langle \gamma_1(a) ,\gamma_1(c)+\gamma_2(d) \rangle \\
& = -2 \langle \gamma_1(a), Z_{c,d} \rangle = -2 \langle Z_{a,b}, Z_{c,d} \rangle.
\end{align*}
On the other hand, by definition of $\Phi$ we have
$$
\Phi((a,b),(c,d))= r_{c,d}^2 - \|Z_{a,b}\|^2-\|Z_{c,d}\|^2 = r_{c,d}^2 - \|Z_{a,b}-Z_{c,d}\|^2-2\langle Z_{a,b}, Z_{c,d} \rangle,
$$
and by plugging the last inequality in this expression we easily obtain
$$
\|Z_{a,b}-Z_{c,d}\|^2 \leq r_{c,d}^2 \quad \text{for all} \quad (c,d) \in E^2.
$$
\end{proof}

Since the preceding Lemma covers the case $r_{a,b}=0$ for some $(a,b)\in E^2,$ we may suppose from this moment on that $r_{a,b}>0$ for all $(a,b)\in E^2.$  Recall that we are also assuming that $E$ is finite. The following Lemma is essentially a restatement (for $P$ a finite set and replacing $\R^n$ with a Hilbert space) of \cite[2.10.40, p. 199]{Federer}, whose proof obviously extends for balls in Hilbert spaces if we bear in mind that they are compact in the weak topology.

\begin{lemma}[Kirszbraun]\label{lemakirszbraun} For every $\lambda \geq 0,$ we denote $\B_{a,b}(\lambda) = B(Z_{a,b}, \lambda r_{a,b}).$ Define $\lambda_0 \geq 0$ as
$$
\lambda_0:= \inf \bigg \lbrace \lambda \geq 0 \: : \: \bigcap_{(a,b)\in E^2} \B_{a,b}(\lambda) \neq \emptyset \bigg \rbrace.
$$
 Then $\bigcap_{(a,b)\in E^2} \B_{a,b}(\lambda_0) = \{ Z_0 \},$ where
$$
Z_0 \in \co \left\lbrace Z_{a,b} \: : \: (a,b) \in E^2 \quad \text{and} \quad \|Z_0-Z_{a,b}\| = \lambda_0 r_{a,b} \right\rbrace.
$$
\end{lemma}
We will finish the proof of Theorem \ref{mainextensiontheorem} by establishing the following.

\begin{lemma} \label{lemmaconclusion}
With the notation of Lemma \ref{lemakirszbraun}, the number $\lambda_0$ satisfies $\lambda_0 \leq 1.$ In particular, the family of balls $\lbrace \B_{a,b} \: : \: (a,b) \in E^2 \rbrace$ has nonempty intersection.
\end{lemma}
\begin{proof}
If we define $\E= \lbrace (a,b) \in E^2 \: : \: \|Z_0-Z_{a,b}\| = \lambda_0 r_{a,b} \rbrace,$ from Lemma \ref{lemakirszbraun} we learn that 
\begin{equation}\label{exressionvm}
Z_0= \sum_{(a,b) \in \E} \xi_{a,b} Z_{a,b} \quad \text{with} \quad \sum_{(a,b)\in \E} \xi_{a,b}=1, \quad \xi_{a,b} \geq 0 \quad \text{for all} \quad (a,b)\in \E.
\end{equation}

By these properties we have that
$$
\sum_{(a,b) \in \E} \xi_{a,b}(Z_0- Z_{a,b})=\sum_{(c,d) \in \E} \xi_{c,d}(Z_0- Z_{c,d})=0,
$$
and therefore
\begin{equation}\label{equationproductzero}
\sum_{(a,b),\: (c,d) \in \E} \xi_{a,b}\xi_{c,d} \langle Z_0- Z_{a,b}, Z_0-Z_{c,d} \rangle=0.
\end{equation}

For any $(a,b), (c,d) \in \E$ we have that $\|Z_{a,b}-Z_0\|^2= \lambda_0^2 r^2_{a,b}$ and $\|Z_{c,d}-Z_0\|^2= \lambda_0^2 r^2_{c,d}$, and it is also clear that
$$
\|Z_{a,b}-Z_{c,d}\|^2 = \|Z_{a,b}-Z_0 \|^2+\|Z_{c,d}-Z_0 \|^2 -2 \langle Z_0- Z_{a,b}, Z_0-Z_{c,d} \rangle.
$$
Hence, multiplying by $\xi_{a,b} \xi_{c,d},$ taking sums over $(a,b), (c,d) \in \E$ and using \eqref{equationproductzero} we obtain
$$
\sum_{(a,b),\: (c,d) \in \E} \xi_{a,b}\xi_{c,d} \|Z_{a,b}-Z_{c,d}\|^2 = \lambda_0^2 \!\! \sum_{(a,b),\: (c,d) \in \E} \xi_{a,b}\xi_{c,d} (r_{a,b}^2+r_{c,d}^2),
$$
Now we set
\begin{align*}
\Delta : & = \sum_{(a,b),\: (c,d) \in \E} \xi_{a,b}\xi_{c,d} \left( -\|Z_{a,b}-Z_{c,d}\|^2 + r_{a,b}^2+r_{c,d}^2 \right) \\
& = (1-\lambda_0^2) \sum_{(a,b),\: (c,d) \in \E} \xi_{a,b}\xi_{c,d} (r_{a,b}^2+r_{c,d}^2).
\end{align*}
Since all the radii $r_{a,b}$ are positive, it is clear that showing $\lambda_0 \leq 1$ is equivalent to $\Delta \geq 0$ 

\begin{claim} \label{claimdeltanonnegative}
$\Delta \geq 0.$ 
\end{claim}
We inmediately see that
$$
\Delta= \sum_{(a,b),\: (c,d) \in \E} \!\!\!\! \xi_{a,b}\xi_{c,d} \left( -\|Z_{a,b}\|^2-\|Z_{c,d}\|^2 + r_{a,b}^2+r_{c,d}^2 \right) + 2 \!\!\!\! \sum_{(a,b),\: (c,d) \in \E} \!\!\!\!\xi_{a,b}\xi_{c,d} \: \langle Z_{a,b}, Z_{c,d} \rangle.
$$
On the other hand, by \eqref{exressionvm} we obtain
$$
\|Z_0\|^2 = \sum_{(a,b),\: (c,d) \in \E} \xi_{a,b}\xi_{c,d} \langle Z_{a,b}, Z_{c,d} \rangle.
$$
This implies that
\begin{equation}\label{newexpressiondelta}
\Delta = 2 \|Z_0\|^2 + \!\!\! \sum_{(a,b),\: (c,d) \in \E} \!\!\! \xi_{a,b}\xi_{c,d}  \: \Phi( (a,b), (c,d)).
\end{equation}
We define now
$$
\Gamma_1:= \sum_{(a,b)\in \E} \xi_{a,b} \gamma_1(a), \quad \Gamma_2:= \sum_{(a,b)\in \E} \xi_{a,b} \gamma_2(b)
$$
and we easily deduce from equation \eqref{sumgammasvab} that
\begin{equation} \label{expressionsumxy}
\Gamma_1+\Gamma_2 = \sum_{(a,b) \in \E} \xi_{a,b} ( \gamma_1(a)+\gamma_2(b))= 2\sum_{(a,b) \in \E} \xi_{a,b}  Z_{a,b} = 2 Z_0.
\end{equation}
Applying Lemma \ref{lemmacrucialinequality} we obtain
\begin{align*}
& \sum_{(a,b),\: (c,d) \in \E} \!\!\! \xi_{a,b}\xi_{c,d} \: \Phi( (a,b), (c,d)) \geq - \!\!\! \!\!\! \sum_{(a,b),\: (c,d) \in \E} \!\!\! \xi_{a,b}\xi_{c,d} \Big ( \langle  \gamma_1(a), \gamma_2(d) \rangle+\langle \gamma_1(c), \gamma_2(b) \rangle \Big ) \\
& \quad =  - \: \Big \langle \sum_{(a,b)\in \E} \xi_{a,b} \gamma_1(a) , \sum_{(c,d)\in \E} \xi_{c,d} \gamma_2(d) \Big \rangle - \: \Big \langle \sum_{(c,d)\in \E} \xi_{c,d} \gamma_1(c) , \sum_{(a,b)\in \E} \xi_{a,b} \gamma_2(b) \Big \rangle \\
& \quad = -\langle \Gamma_1,\Gamma_2 \rangle  -\langle \Gamma_1,\Gamma_2 \rangle =  -2\langle \Gamma_1,\Gamma_2 \rangle.
\end{align*}
Combining this inequality with equations \eqref{newexpressiondelta} and \eqref{expressionsumxy} we have
$$
\Delta \geq 2 \: \bigg \| \frac{ \Gamma_1+\Gamma_2}{2} \bigg \|^2 - 2 \langle \Gamma_1,\Gamma_2 \rangle = 2 \: \bigg  \| \frac{ \Gamma_1-\Gamma_2}{2} \bigg \|^2,
$$
which implies $\Delta \geq 0.$ This finishes the proof of Claim \ref{claimdeltanonnegative}, and therefore that of Lemma \ref{lemmaconclusion} too.
\end{proof}

The proofs of Theorems \ref{mainextensiontheorem} and \ref{lipschitzversionmainextensiontheorem} are now complete.

\medskip

Let us finish this paper by showing that there exist bounded, smooth convex functions defined on an open neighborhood of a closed ball in $X:=\ell_2(\R)$  which have no continuous convex extensions to all of $X$.
Denote by $C$ the closed unit ball of $X.$ The natural complexification of the space is 
$
X_\C = \ell_2(\C).
$
Also let $U=\{x\in X : \|x\|<2\}$, $U_\C=\{ x\in X_\C : \|x\|<2\}$, and $S_X=\{x\in X : \|x\|=1\}.$ 
\begin{example}\label{counterexample in Hilbert}
There exists a function $F: U \to \R$ such that
\begin{itemize}
\item[(i)] $F$ is analytic on $U$;
\item[(ii)] $F$ is convex on $U$ with $D^2F(x)(v^2) \geq 1$ for every $x\in U$, $v\in S_X$;
\item[(iii)] $F$ is bounded on $C$, and
\item[(iv)] $F_{|_C}$ has no continuous convex extension to the whole space $X.$
\end{itemize}
\end{example}
\begin{proof}
Let $\{e_n\}_{n\in\N}$ be the canonical basis of $X$, and consider the sequence of vectors $ \lbrace \widetilde{e_n} \rbrace_n \subset C$ defined as follows:
$$
\widetilde{e_n}=\frac{1}{2}e_1+ \frac{\sqrt{3}}{2}e_n, \quad n\geq 2.
$$
For every $n\geq 2,$ we define the linear functional $h_n \in X^*$ by $h_n(x)= \langle x, \widetilde{e_n} \rangle$ for all $x\in X.$ Equivalently, for every $x=(x_n)_{n\geq 1} \in X,$ we have $h_n(x)= \frac{1}{2}x_1+ \frac{\sqrt{3}}{2}x_n$ for every $n\geq 2.$ Now let us define
$$
\begin{array}{rccl}
f : &U & \longrightarrow &  \R  \\
   &x & \longmapsto  & \sum_{n=2}^\infty (h_n(x))^{2n},
\end{array}
$$
or equivalently $f(x) = \sum_{n \geq 2} \left( \frac{1}{2}x_1+ \frac{\sqrt{3}}{2}x_n \right)^{2n}$ for all $x=(x_n)_n\in U.$ Let us first check that $f$ is well defined. Given $x\in U,$ take $r=2-|x_1|>0.$ Because $x\in \ell_2,$ there is some $n_0\in \N$ such that $|x_n| \leq \frac{r}{2 \sqrt{3}}$ whenever $n \geq n_0.$ Therefore, if $n \geq n_0,$ we have
\begin{align*}
\bigg| \frac{1}{2}x_1+ \frac{\sqrt{3}}{2}x_n \bigg| & \leq \frac{1}{2}|x_1|+ \frac{\sqrt{3}}{2} |x_n| = \frac{1}{2}(2-r)+ \frac{\sqrt{3}}{2}|x_n| \\
& \leq  \frac{1}{2}(2-r)+ \frac{r}{4}= 1- \frac{r}{4}=: \lambda. 
\end{align*}
Since $\lambda <1,$  
$$
\sum_{n\geq n_0} \bigg| \frac{1}{2}x_1+ \frac{\sqrt{3}}{2}x_n \bigg| ^{2n} \leq \sum_{n\geq n_0} \lambda^{2n}
$$
converges and this shows that $f(x)$ is finite.
\begin{claim}
$f$ is bounded by $M:=\frac{49}{24}$ on $C$.
\end{claim}
\begin{proof} Given $x\in C,$ and $x=(x_n)_{n \geq 1},$ since $\sum_{n \geq 1} x_n^2 \leq 1,$ we have that $\sum_{n \geq 2} x_n^2 \leq 1- x_1^2;$ and this implies that there is at most one coordinate $N \geq 2 $ such that $x_N^2 > \frac{1-x_1^2}{2}.$ Hence, the rest of the coordinates satisfy
$$
 |x_n| \leq \sqrt{\frac{1-x_1^2}{2}} \quad \text{for every} \quad n \geq 2 \quad \text{with} \quad n \neq N.
 $$ 
And, of course, $|x_N| \leq \sqrt{1-x_1^2}.$ 
We easily have
\begin{align*} 
f(x) & \leq \sum_{n\geq 2} \left( \frac{1}{2}|x_1|+ \frac{\sqrt{3}}{2}|x_n| \right)^{2n} = \left( \frac{1}{2}|x_1|+ \frac{\sqrt{3}}{2}|x_N| \right)^{2N}
\!\!\! + \!\!\! \sum_{n\geq 2, \: n \neq N} \left( \frac{1}{2}|x_1|+ \frac{\sqrt{3}}{2}|x_n| \right)^{2n} \\
& \leq \left( \frac{1}{2}|x_1|+ \frac{\sqrt{3}}{2}\sqrt{1-x_1^2} \right)^{2N} + \sum_{n\geq 2, \: n \neq N} \left( \frac{1}{2}|x_1|+ \frac{\sqrt{3}}{2}\sqrt{\frac{1-x_1^2}{2}} \right)^{2n}.
\end{align*}
In order to get a bound for the first sum in the last term, we consider the function $g(t) = \frac{t}{2}+\frac{\sqrt{3}}{2} \sqrt{1-t^2}, \: t\in [0,1].$ A simple calculation shows that $g$ has a maximum at $t=\frac{1}{2}$ and then $g(t) \leq g(1/2)=1$ for all $t\in [0,1].$ Therefore
$$
\left( \frac{1}{2}|x_1|+ \frac{\sqrt{3}}{2}\sqrt{1-x_1^2} \right)^{2N} \leq 1.
$$
The second sum can be bounded as follows. Take $h(t)= \frac{t}{2}+ \frac{\sqrt{3}}{2}\frac{\sqrt{1-t^2}}{\sqrt{2}}, \: t\in [0,1].$ We easily deduce that $h$ attains a maximum at $t= \sqrt{\frac{2}{5}}.$ Hence 
$
h(t) \leq h\left( \sqrt{\frac{2}{5}} \right)= \sqrt{\frac{5}{8}},
$
for every $t\in [0,1].$ This implies
$$
\left( \frac{1}{2}|x_1|+ \frac{\sqrt{3}}{2}\sqrt{\frac{1-x_1^2}{2}} \right)^{2n} \leq \left( \sqrt{\frac{5}{8}} \right)^{2n} = \left( \frac{5}{8} \right)^n \quad \text{for all} \quad n\geq 2, \quad n \neq N.
$$
Therefore, 
$
f(x) \leq 1 + \sum_{n\geq 2, \: n \neq N} \left( \frac{5}{8} \right)^n  \leq 1 + \sum_{n\geq 2} \left( \frac{5}{8} \right)^n  = \frac{49}{24}.
$
\end{proof}
\begin{claim}
$f$ is real analytic on $U.$
\end{claim} 
\begin{proof} Consider the complex function
$$
\begin{array}{rccl}
\widetilde{f} : &U_\C & \longrightarrow &  \C  \\
   &z & \longmapsto  & \sum_{n=2}^\infty \left( \frac{1}{2}z_1+ \frac{\sqrt{3}}{2}z_n \right)^{2n}
\end{array}
$$
Obviously the restriction of $\widetilde{f}$ to $U$ is the function $f,$ and we can see that $\tilde{f}$ is well defined with the same calculations as we made above for $f.$ Of course it is enough to prove that $\widetilde{f}$ is holomorphic on $U_\C$, for which in turn it is enough to check that, given $z\in U_\C$ there are $r>0$ and a sequence $\lbrace M_n \rbrace_{n\geq 2} $ of positive numbers such that
$$
\sum_{n \geq 2} M_n < + \infty \quad \text{and} \quad \bigg|  \frac{1}{2}y_1+ \frac{\sqrt{3}}{2}y_n \bigg| ^{2n} \leq M_n \quad \text{for all} \quad y\in \overline{B}_\C(z,r) \subseteq U_\C, \quad n \geq 2,
$$
where $\overline{B}_\C(z,r) = \lbrace x\in X_\C \: : \: \|z-y\| \leq r \rbrace.$
Indeed, fix $z\in U_\C.$ We take $r>0$ such that $\overline{B}_\C(z,r) \subset U_\C $ with $\| z\| + r <2$ and $r\leq \frac{2-|z_1|}{4(1+\sqrt{3})}.$
 Find $n_0\in \N$ such that $|z_n| \leq \frac{2-|z_1|}{2 \sqrt{3}}$ whenever $n \geq n_0.$ Of course these $r>0$ and $n_0\in \N$ only depend on $z.$ Define the numbers
$$
 \lambda_n   =  \left\lbrace 
	\begin{array}{ccl}
	1+\sqrt{3}  & \mbox{if } &  2\leq n \leq n_0-1 \\
	\dfrac{6+|z_1|}{8}  & \mbox{if }&  n \geq n_0,
	\end{array}
	\right.
$$
and $M_n= \lambda_n^{2n}$ for all $n\geq 2.$ Since $|z_1| <2,$ the sum $\sum_{n\geq 2} M_n $ converges. If $y\in \overline{B}_\C(z,r) ,$ with $y=(y_n)_{n\geq 1},$ then
$
|y_n|\leq r + |z_n| 
$
for every $n\geq 1.$ Therefore, if $n \geq n_0,$ because $|z_n| \leq \frac{2-|z_1|}{2 \sqrt{3}}$ and $r\leq \frac{2-|z_1|}{4(1+\sqrt{3})}$ we have
\begin{align*}
\bigg|  \frac{1}{2}y_1 + \frac{\sqrt{3}}{2}y_n \bigg|  &\leq \frac{1}{2}|y_1|+ \frac{\sqrt{3}}{2}|y_n| \leq \frac{1}{2}(|z_1|+r)+ \frac{\sqrt{3}}{2}(|z_n|+r) \\
& \leq \frac{1+\sqrt{3}}{2} \frac{2-|z_1|}{4(1+\sqrt{3})} + \frac{|z_1|+\frac{1}{2}(2-|z_1|)}{2}= \lambda_n.
\end{align*}
And for integers $2\leq n \leq n_0-1,$ we have the obvious inequality $
 \big |  \frac{1}{2}y_1 + \frac{\sqrt{3}}{2}y_n \big | \leq 1+ \sqrt{3} = \lambda_n.
$
Hence
$$
\bigg|  \frac{1}{2}y_1 + \frac{\sqrt{3}}{2}y_n \bigg|^{2n} \leq M_n \quad \text{for every}  \quad n \geq 2 
$$
and this proves our statement. 
\end{proof} 

Now, the convexity of $f$ can be easily checked: The function $f_n=g_n \circ h_n,$ being $h_n$ a linear functional and $\R \ni t \to g_n(t)=t^{2n}$ a convex function for all $n \geq 2,$ is convex on $U,$ and $f,$ being the sum of convex functions, is convex on $U$ as well. Now define $F:=f+ N,$ where $N: X \to \R$ is the function defined by $N(x)= \frac{\|x \|^2}{2}$ for all $x\in X.$ Since $X$ is a Hilbert space, the function $N$ is analytic on $X.$ Of course $N$ is bounded on $C$ and $D^2 N (x)(v)^2 =   \|v\|^2 = 1$ for all $v\in S_X$ and all $x\in X.$ Hence $F$ is real analytic, is bounded on $C$ and, since $f$ is convex and differentiable, $D^2F(x)(v^2)= D^2f(x)(v^2)+ D^2N(x)(v^2) \geq 1$ for all $x\in U$ and all $v\in S_X.$ We then have proved (i), (ii) and (iii) of our Theorem. 

In order to prove (iv), consider the minimal convex extension of $F$,
$$
m_C (F) ( x)= \sup_{y\in C} \lbrace F(y)+ \langle \nabla F (y) , x-y \rangle \rbrace, \quad x\in X.
$$ Observe that (iv) will be proved as soon as we find points $x\in X$ with $m_C(F)(x)= +\infty. $ We next prove that in fact $m_C(F)=+\infty$ for all $x$ of the form $x=r e_1,$ with $r>2.$ So fix $r>2$ and $x=re_1.$ For any $k \geq 2$ and $n \geq 2$ we inmediately see that $\langle \widetilde{e_n} , \widetilde{e_k} \rangle = 1/4 $ for $n\neq k$ and $\langle \widetilde{e_k} , \widetilde{e_k} \rangle = 1. $ Then
$$
f(\widetilde{e_k})= 1+\sum_{n\geq 2, \: n \neq k} \left( \frac{1}{4} \right)^{2n} \quad \text{and} \quad N(\widetilde{e_k}) = \frac{1}{2}, \quad k \geq 2.
$$
Since $f$ is analytic, we can calculate its derivatives by differentiating the series term by term, and then
$$
\langle \nabla f (\widetilde{e_k}) , v  \rangle= \sum_{ n \geq 2 } 2n  \langle \widetilde{e_k} , \widetilde{e_n} \rangle ^{2n-1}\langle v , \widetilde{e_n} \rangle = \sum_{ n \geq 2, \: n \neq k } 2n  \left( \frac{1}{4} \right) ^{2n-1} \! \! \! \langle v , \widetilde{e_n} \rangle + 2k \: \langle v, \widetilde{e_k} \rangle
$$
for every $v\in X$ and $k \geq 2.$ On the other hand,  $\langle \nabla N (\widetilde{e_k}) , v  \rangle = \langle \widetilde{e_k}, v \rangle$  for all $v\in X.$ For $v=x-\widetilde{e_k},$ we have
$$
 \langle v, \widetilde{e_n} \rangle = \bigg \langle r e_1, \frac{1}{2}e_1+\frac{\sqrt{3}}{2} e_n \bigg \rangle - \langle \widetilde{e_k} , \widetilde{e_n} \rangle = \frac{r}{2}-\left\lbrace 
	\begin{array}{ccl}
	1  & \mbox{if } &  n=k \\
	\frac{1}{4}  & \mbox{if }& n \neq k,
	\end{array}
	\right.
$$
Gathering the above inequalities we obtain, for $k \geq 2,$
\begin{align*}
& F(\widetilde{e_k})   +  \langle \nabla F (\widetilde{e_k} ) , x-\widetilde{e_k}  \rangle =  f(\widetilde{e_k})+N(\widetilde{e_k}) +  \langle \nabla f (\widetilde{e_k} ) , x-\widetilde{e_k}  \rangle +   \langle \nabla N (\widetilde{e_k} ) , x-\widetilde{e_k}  \rangle\\
& = 1+\sum_{n\geq 2, \: n \neq k} \! \! \! \left( \frac{1}{4} \right)^{2n} + \frac{1}{2}+\sum_{ n \geq 2, \: n \neq k } \! \! \! 2n  \left( \frac{1}{4} \right) ^{2n-1} \! \! \! \left( \frac{r}{2}-\frac{1}{4} \right) + 2k \left( \frac{r}{2}-1 \right)+ \left( \frac{r}{2}- 1 \right) \\
& \geq k (r-2);
\end{align*}
and the last term tends to $+\infty$ as $ k $ goes to $+\infty.$ We thus have proved that $m_C(F)(x)=+\infty$ for those points $x\in X$ of the form $x=r e_1, \: r>2.$
\end{proof}


\end{document}